\documentclass[11pt, a4paper, reqno]{amsart}
\usepackage{graphicx}
\usepackage{amsthm, amsmath, amssymb}

\theoremstyle{plain}
\newtheorem{thm}{Theorem}
\newtheorem{thm*}{Theorem}

\newtheorem{lem}[thm]{Lemma}

\theoremstyle{definition}
\newtheorem*{defn}{Definition}
\newtheorem*{defn*}{Definition}

\newtheorem*{conj*}{Conjecture}
\newtheorem*{question*}{Question}

\theoremstyle{remark}
\newtheorem*{rem}{Remark}
\newtheorem*{rem*}{Remark}

\newcommand{\C}{\mathbb{C}}
\newcommand{\CC}{\widehat {\mathbb{C}}}
\newcommand{\R}{\mathbb{R}}
\newcommand{\Z}{\mathbb{Z}}

\newcommand{\D}{\mathbb{D}}
\newcommand{\HH}{\mathbb{H}}
\newcommand{\E}{\mathbb{E}}

\newcommand{\eps}{\varepsilon}
\newcommand{\vphi}{\varphi}
\renewcommand{\phi}{\varphi}

\renewcommand{\epsilon}{\varepsilon}

\newcommand{\sminus}{\smallsetminus}
\newcommand{\calC}{\mathcal{C}}

\renewcommand{\Re}{\operatorname{Re}}
\renewcommand{\Im}{\operatorname{Im}}
\newcommand{\supp}{\operatorname{supp}}
\newcommand{\Area}{\operatorname{Area}}
\newcommand{\Jac}{\operatorname{Jac}}
\newcommand{\Res}{\operatorname{Res}}
\newcommand{\Cr}{\operatorname{Cr}}
\newcommand{\tinyhsp}{\hspace{.06667em}}

\title{Conformality of quasiconformal mappings at a point, revisited}
\author{Mitsuhiro Shishikura}
\address{Department of Mathematics, Kyoto University, 
Kyoto 606-8502, Japan. }
\thanks{This work was partially supported by JSPS Grant-in-Aid 26287016 and 15K13444.}
\keywords{quasiconformal mapping} 
\subjclass[2010]{30C62}

\begin{document}
\maketitle

\begin{abstract} 
We present a new and simple proof of Teichm\"uller-Wittich-Belinski\u \i's and Gutlyanski\u\i-Martio's theorems on the conformality of 
quasiconformal mappings at a given point.  Known proofs gave separate estimates for the radial and angular variations, 
but our proof unifies them using Gr\"otzsch-type inequality for the variation of cross-ratio of four points on the Riemann sphere.  
We also give a sufficient condition for $C^{1+\alpha}$-conformality 
\end{abstract}

\section*{Introduction}

Quasiconformal mappings are known to be differentiable almost everywhere with respect to the Lebesgue measure 
(see \cite{AhlforsQC}, \cite{LehtoVirtanen}).  
However if one picks a specific point, then the differentiability is not guaranteed.  
In this paper, we discuss the conformality 
(i.e. the differentiability with zero $\overline{z}$-derivative) of quasiconformal mappings at a given point.  

\begin{defn} 
For a quasiconformal mapping $f: \C \to \C$, we denote 
\begin{equation*}
\mu_f(z)  = \frac{f_{\bar z}(z)}{f_z(z)} 
= \frac{\  \frac{\partial f}{\partial \bar z}\  }{\frac{\partial f}{\partial z}}, \ \ 
K_f(z) = \frac{1+|\mu_f(z)|}{1-|\mu_f(z)|} 
\ \text{ and } \ 
K(f) = \operatorname{ess}\sup K_f(z).  
\end{equation*}

We say that $f$ is {\it conformal at} $z=z_0$ if the limit 
\begin{equation*} 
f'(z_0) = \lim_{z \to z_0} \frac{f(z)-f(z_0)}{z-z_0}
\end{equation*}
exists and is non-zero.
For simplicity, we only discuss the conformality at $z=0$, but the conformality 
at other points can be treated similarly by translating the coordinate.    
\end{defn}
There is a well-known criterion for the pointwise conformality: 

\begin{thm}[Teichm\"uller \cite{Teichmuller}, Wittich \cite{Wittich}, Belinski\u\i\,  \cite{Belinskii}, Lehto \cite{LehtoDiff}; see \cite{LehtoVirtanen} 
Theorem 6.1] \label{thm:TWB}
If $f$ is a quasiconformal mapping satisfying 
\begin{equation} \label{eq:TWBL} 
\frac {1}{2\pi} \iint_{|z|<r} \frac{|\mu_f(z)|}{|z|^2} dx\tinyhsp dy < \infty \ \text{ for some } r< \infty,  
\end{equation}
then $f$ is conformal at $z=0$. 
\end{thm}

This theorem was improved by: 
\begin{thm} [Gutlyanski\u\i-Martio \cite{GutMartio}] 
\label{thm:GutMartio} 
Let $f: \C \to \C$ be a quasiconformal mapping.  
If 
\begin{equation} \label{eq:muSquareIntegrability} 
\iint_{|z|<1} \frac{|\mu_f(z)|^2}{1-|\mu_f(z)|^2} \, \frac{dx \tinyhsp dy}{|z|^2} < \infty   
\end{equation}
and the limit 
\begin{equation} \label{eq:muIntegrability} 
\lim_{ r \searrow 0} \iint_{r<|z|<1} \frac{\mu_f(z)}{1-|\mu_f(z)|^2} \, \frac{dx\tinyhsp dy}{z^2} 
\end{equation}
exists, then $f$ is conformal at $z=0$.  
\end{thm}

The goal of this paper is to give a new and simple proof of this theorem.  
(Note that in \cite{GutMartio}, it was assumed that \eqref{eq:muSquareIntegrability} holds without $1-|\mu_f(z)|^2$ in the denominator, 
but this is equivalent for a qc-mapping.)
The proof of Theorem \ref{thm:TWB} consists of 
the differentiability of the absolute value $|f(z)|$ 
(e.g. Teichm\"uller \cite{Teichmuller}, Wittich \cite{Wittich}; see \cite{LehtoVirtanen} 
Lemma 6.1), and the estimate the variation of $\arg \frac{f(z)}{z}$ 
(e.g. Belinski\u\i\,  \cite{Belinskii}, Lehto \cite{LehtoDiff}; see \cite{LehtoVirtanen} Lemma 6.2).  
The proof of Theorem \ref{thm:GutMartio} in \cite{GutMartio} also gave the estimates for the absolute value and the argument.  
\par
Our approach 
unifies the two estimates 
into the form of the variation of cross-ratio  of four points $0, z_1, z_2, \infty$, via Cauchy's criterion (see Lemmas \ref{lem:CharacterizeConformality} and \ref{lem:LogVsPoincareDist}).  
The effect of quasiconformal mapping is usually measured by the integral of $\mu_f$ paired with a suitable quadratic differential.  
In our case, the quadratic differential to consider is $\vphi_{z_1,z_2 \vphantom{|}}(z)dz^2$, where 
\begin{equation} \label{eq:DefnPhi} 
\vphi_{z_1,z_2 \vphantom{|}}(z) = \frac {z_1}{z(z-z_1)(z-z_2)}.
\end{equation}
The quasiconformal variation of cross-ratio is formulated in Theorem \ref{thm:FundIneq}, and
the Main Theorem \ref{thm:MainThm} is stated in terms of the integral $J(\mu; z_1, z_2)$ defined by \eqref{eq:DefnJ} using $\vphi_{z_1,z_2 \vphantom{|}}$.  
Heuristically when $|z_2| \ll |z_1|$, in the annular region in-between, $|z_2| \ll |z| \ll |z_1|$, 
the quadratic differential $\vphi_{z_1,z_2 \vphantom{|}}(z)dz^2$ ``looks like'' 
$c\frac{dz^2}{z^2}$, and this explains the appearance of $\frac{1}{z^2}$ in Theorems \ref{thm:TWB} and \ref{thm:GutMartio}.  
(See \cite{HSS}, for a decomposition theorem of quadratic differentials, in which this idea was extensively used.) 
This observation will be justified by the estimates on integrals (Lemmas \ref{lem:EstimateOnJ} and \ref{lem:EstimateOnIps}) 
via the decomposition \eqref{eq:QuadDiffDecomposition}.

Moreover we can also derive a more quantitative estimate on the remainder term: 
\begin{thm} \label{thm:HoelderConformal} 
Let $f:\C \to \C$ be a quasiconformal mapping and suppose that 
\begin{equation} \label{eq:HoelderConformalityAssumption}
I(r) = \iint_{\{z: |z|<r\}} \frac{|\mu_f(z)|}{1-|\mu_f(z)|^2} \frac{dx \tinyhsp dy}{|z|^2} 
\text{  is finite and has order   }  
O(r^\beta) \ (r \searrow 0) 
\text{   for some $\beta>0$. }
\end{equation}
Then 
for any $0<\alpha<\frac{\beta}{2+\beta}$, $f$ is $C^{1+\alpha}$-conformal at $0$ in the sense that  
\begin{equation} \label{eq:HoelderConformality} 
f(z)= f(0) + f'(0)z + O(|z|^{1+\alpha}) 
\quad \text{ as } z \to 0.    
\end{equation}
\end{thm}

\begin{rem} 
Schatz \cite{Schatz} obtained a similar result by assuming a stronger conditipn $\iint \frac{|\mu(z)|}{|z|^p}dx\,dy<\infty$ ($p>2$), 
which implies \eqref{eq:HoelderConformalityAssumption} with $\beta=p-2$. 
McMullen \cite{McMullenRenorm2} (Theorem 2.25)
obtained the same conclusion by assuming $\Area(B_r(0) \cap \supp \mu )= O(r^{2+\alpha})$, which is 
again stronger.  
\par
\end{rem}

For further references, see also \cite{ReichWalczak}, \cite{Drasin}, \cite{BrakalovaJenkins}.

The author would like to thank Kari Astala, David Drasin, Frederick Gardiner, Anatoly Golberg for helpful discussions.  

\section{Conformality at $z=0$ and cross-ratio}

We start with nothing but Cauchy's criterion.  
Define the cylinder $\calC=\C/2\pi i \Z$ and its distance $|w|_{\calC}=\inf \{ |w+2\pi i n| : \, n \in \Z\}$.

\begin{lem} \label{lem:CharacterizeConformality}  
Let $f:\C \to \C$
be an orientation-preserving homeomorphism with $f(0)=0$ and fix a constant $0<\delta_1 \le 1$.  
Then the following are equivalent: 
\begin{enumerate}
\renewcommand{\labelenumi}{{\rm (\alph{enumi})}}
\item $f$ is conformal at $z=0$;
\item there exists a limit $\lim_{z \to 0} \log \frac{f(z)}{z}$ for a suitable choice of branch of $\log$;
\item for any $\eps>0$, there exists $r>0$ such that 
if $0<|z_1|<r$ and $0<|z_2| \le \delta_1 |z_1|$, then 
\begin{equation} \label{eq:CauchyCriterion} 
\left| \, \log \frac{f(z_1)}{z_1} - \log \frac{f(z_2)}{z_2} \right|_{\calC} < \eps.     
\end{equation}
\end{enumerate}

\begin{proof} The most of implications are obvious, and we only prove that (c) implies (b).  
First take $\eps=\frac{\pi}{2}$, then (c) implies that the variation of the argument of $\frac{f(z)}{z}$ 
is less than $\pi$ when $z$ is small.  In such a case,  the distance $|\cdot|_{\calC}$ in \eqref{eq:CauchyCriterion} can be replaced by 
the Euclidean distance.  Now take smaller $\eps$, and let $r$ be as in (c).  
If $|z_1|, |z_2|<r$, then take the third point $z_3$ so that $|z_3| \le \delta_1 |z_1|$, $|z_3| \le \delta_1 |z_2|$, 
then $\left| \, \log \frac{f(z_1)}{z_1} - \log \frac{f(z_2)}{z_2} \right|_{\calC} 
\le \left| \, \log \frac{f(z_1)}{z_1} - \log \frac{f(z_3)}{z_3} \right|_{\calC} + \left| \, \log \frac{f(z_2)}{z_2} - \log \frac{f(z_3)}{z_3} \right|_{\calC} < 2 \eps$.  
By Cauchy's criterion, we have (b).  
(In fact, for any sequence $z_n \to 0$,  $\left\{ \log \frac{f(z_n)}{z_n} \right\}$ will be a Cauchy sequence in $\calC$, hence it is convergent, and this implies 
the convergence of $\log \frac{f(z)}{z}$ as $z \to 0$.)
\end{proof}

\end{lem}

\begin{defn} 
For distinct points $z_1, z_2, z_3, z_4$ in $\C$, define the {\it cross-ratio} by 
\begin{equation*} 
\Cr( z_1, z_2, z_3, z_4 ) = 
\frac{z_1 - z_3}{z_2-z_3} \cdot \frac{z_2-z_4}{z_1 - z_4}.  
\end{equation*}
This definition extends to the case where one of $z_j$'s is $\infty$ by taking the limit.  
The cross-ratio belongs to the three punctured sphere 
$\Omega := \CC \sminus \{0, 1, \infty\} = \C \sminus \{0,1\}$.  Denote the 
hyperbolic distance on $\Omega$ by $d_\Omega(\cdot, \cdot)$, 
which is induced from $\frac{|dz|}{\Im z}$ on the universal cover $\HH$.  
\end{defn}

\smallskip

Let $f$ be as in Lemma \ref{lem:CharacterizeConformality} and take $z_1, z_2 \in \C \sminus \{0\}$ with $z_1 \neq z_2$.  
Denote $\zeta_1 = \frac{z_2}{z_1} = \Cr(z_1, z_2, \infty, 0)$ and 
$\zeta_2 = \frac{f(z_2)}{f(z_1)} = \Cr(f(z_1), f(z_2),\infty, 0)$.  
We need to estimate 
\begin{equation*} 
\left| \log \frac{f(z_1)}{z_1} - \log \frac{f(z_2)}{z_2} \right|_{\calC} 
= \left| \log \zeta_1 - \log \zeta_2 \right|_{\calC}  
= \left| \log \Cr(z_1, z_2, \infty, 0) - \log \Cr(f(z_1), f(z_2), \infty, 0) \right|_{\calC}.   
\end{equation*}

\begin{lem} 
\label{lem:LogVsPoincareDist} 
For any $L > 0$, there exist constants $C_1>0$ and $0<\delta_1<1$ such that 
if $\zeta_1, \zeta_2 \in \Omega = \C \sminus \{0,1\}$ satisfy $|\zeta_1|<\delta_1$ 
and $d_\Omega(\zeta_1,\zeta_2) \le L$,  
then 
\begin{equation*} 
\left| \log \zeta_1 - \log \zeta_2 \right|_{\calC} 
\le C_1 d_\Omega(\zeta_1,\zeta_2) \cdot \log \frac{1}{|\zeta_1|}. 
\end{equation*}
\end{lem}

\begin{proof} Let $\rho_\Omega(\zeta)|d\zeta|$ be the hyperbolic metric of $\Omega$.  
It is well-known (see \cite{AhlforsCI} \S 1-8) that there exist $0<\delta_0<1$ and $C_0>0$ such that 
\begin{equation*} 
\rho_\Omega(\zeta) \ge \frac{C_0}{|\zeta| \log \frac {1}{|\zeta|}} \ \text{ for } 0 < |\zeta| \le \delta_0. 
\end{equation*}
Let $\nu = e^{L/C_0} \, (>1)$.  
Then for $0 < r \le \delta_0$, the distance between $\{\zeta: |\zeta|=r\}$ and $\{\zeta': |\zeta'|=r^\nu\}$ 
is bounded below by 
\begin{equation*} 
d_\Omega(\zeta, \zeta') \ge C_0 \int_{r^\nu}^r \frac{ds}{s\log\frac{1}{s}}
=C_0 \left(-\log\log\frac{1}{r} + \log\log\frac{1}{r^\nu}\right)=C_0 \log \nu=L.  
\end{equation*}
Let $\delta_1 = \delta_0^{\nu}$ and $C_1= \frac{\nu}{C_0}$.  
Suppose $0<|\zeta_1| \le \delta_1$ and $d_\Omega(\zeta_1, \zeta_2) \le L$, 
and let $\gamma$ be the shortest hyperbolic geodesic in $\Omega$ 
joining $\zeta_1$ and $\zeta_2$.  Then, by the above estimate for the circles of radii $|\zeta_1|^{\nu}$, $|\zeta_1|$, $|\zeta_1|^{1/\nu}$,  
we have for $\zeta \in \gamma$, 
$|\zeta_1|^\nu \le |\zeta| \le |\zeta_1|^{1/\nu} \le \delta_0$.  Hence 
\begin{align*}
\left| \log \zeta_1 - \log \zeta_2 \right|_{\calC} 
&\le \left| \int_{\gamma} \frac {d\zeta}{\zeta} \right| \le \int_{\gamma} \frac {|d\zeta|}{|\zeta|}
\le \frac{1}{C_0}\int_{\gamma} \rho_\Omega(\zeta) \log \frac{1}{|\zeta|} |d\zeta| \\
&\le \frac{\nu}{C_0} \log \frac{1}{|\zeta_1|} \int_{\gamma} \rho_\Omega(\zeta) |d\zeta|
= C_1 d_\Omega(\zeta_1,\zeta_2) \cdot \log \frac{1}{|\zeta_1|}.
\end{align*}
\vskip -9mm
\end{proof}

Thus, in order to to show the conformality,  we want to show that 
 \  $d_\Omega(\zeta_1, \zeta_2) \cdot \log \frac {1}{|\zeta_1|}$ \  is small 
when $z_1$, $z_2$ are small.  

\section{Gr\"otzsch-type inequality for cross-ratio variation}

We need the following Gr\"otzsch-type inequality for cross-ratio variation.  

\begin{thm} \label{thm:FundIneq} Let $f: \CC \to \CC$ be a quasiconformal mapping and 
$z_1, z_2, z_3, z_4$ distinct points in $\CC$, and put $z_j'=f(z_j)$ ($j=1,2,3,4$).  
Then 
\begin{equation} \label{eq:FundIneq} 
d_{\Omega}(\Cr(z_1, z_2, z_3, z_4), \Cr(z_1',z_2',z_3',z_4')) 
\le \log \overline K_f(z_1, z_2, z_3, z_4), 
\end{equation}
where 
\begin{equation} \label{eq:Kfourpoints}
 \overline K_f(z_1, z_2, z_3, z_4) 
:= 
\frac {\, \displaystyle \sup_{\theta \in \R} 
\iint_\C \frac {\, \left|1+ e^{i\theta} \mu(z) \frac{\vphi(z)}{|\vphi(z)|}\right|^2}{1-|\mu(z)|^2} |\vphi(z)| dx\tinyhsp dy \, }
{\displaystyle \iint_\C |\vphi(z)| dx\tinyhsp dy}
\end{equation}
with $\mu(z)=\mu_f(z)$ and $\vphi(z)=\frac{1}{(z-z_1)(z-z_2)(z-z_3)(z-z_4)}$ (omit $(z-z_j)$ if $z_j=\infty$).  
\end{thm}

This is a special case of Fundamental Inequality in the Teichm\"uller theory 
(\cite{GardinerLakic} Chap. 4, Theorem 9) applied to four punctured sphere.  
In fact, this case can be proven directly as in \cite{AhlforsQC}. 
For the completeness, we will outline this proof in Appendix A.  \par
Note that the above inequality implies the classical Gr\"otzsch inequality \\
$d_{\Omega}(\Cr(z_1, z_2, z_3, z_4), \Cr(z_1',z_2',z_3',z_4')) \le \log K(f)$, 
since $\overline K_f(z_1, z_2, z_3, z_4) \le K(f)$.  
We now express $\overline K_f$ in terms of the following integrals.  

\begin{defn*} Let $z_1, z_2 \in \C \sminus \{0\}$ with $z_1 \neq z_2$. Note that $|\vphi_{z_1,z_2 \vphantom{|}}|$ is integrable over $\C$. 
Let 
\begin{equation} \label{eq:DefnJstar} 
J_*(z_1, z_2) = 
\iint_\C |\vphi_{z_1,z_2 \vphantom{|}}(z)| dx \tinyhsp dy.  
\end{equation}
For a measurable function $\mu: \C \to \C$ with $||\mu(z)||_\infty<1$, 
define 
\begin{equation} \label{eq:DefnJ} 
J(\mu; z_1, z_2) = 
2 \left| \iint_\C \frac{\mu(z)\vphi_{z_1,z_2 \vphantom{|}}(z)}{1-|\mu(z)|^2}dx \tinyhsp dy\right|
+ 2 \iint_\C \frac{|\mu(z)|^2 |\vphi_{z_1,z_2 \vphantom{|}}(z)|}{1-|\mu(z)|^2}dx \tinyhsp dy.  
\end{equation}
\end{defn*}

\begin{lem} \label{lem:JandKbar} 
Suppose $0<|z_2|< |z_1|$. 
Then for $f$ in Theorem  \ref{thm:FundIneq}, we have 
\begin{gather*}
J_*(z_1,z_2) \ge 2 \pi \frac{1}{\,|1- \frac{z_2}{z_1}|\,}  \log \frac {\,|z_1|\,}{|z_2|}; \\
\overline K_f(z_1,z_2,0,\infty) = 1+ \frac{J(\mu_f; z_1, z_2)}{J_*(z_1, z_2)}; \\
J(\mu_f; z_1, z_2) 
\le \iint_\C (K_f(z) -1) |\vphi_{z_1,z_2 \vphantom{|}}(z)| dx \tinyhsp dy \le (K(f)-1)J_*(z_1.z_2).
\end{gather*}
\end{lem}

\begin{proof}
Denote $\vphi=\vphi_{z_1,z_2}$. 
By the Residue Theorem, we have for $|z_2|<r<|z_1|$, 
\begin{equation*}
\int_{|z|=r} z \tinyhsp \vphi(z) dz = \int_{|z|=r} \frac {z_1}{(z-z_1)(z-z_2)} dz 
= 2\pi i \Res_{z=z_2} \frac {z_1}{(z-z_1)(z-z_2)} = 2 \pi i \frac {z_1}{z_2-z_1}. 
\end{equation*} 
Hence 
$2\pi \left| \frac{z_1}{z_1-z_2} \right| 
= \left| \int_{0}^{2\pi} r e^{i\theta} \vphi(r e^{i\theta}) i r e^{i\theta} d\theta \right| 
\le r^2 \int_{0}^{2\pi} |\vphi(r e^{i\theta})| d\theta$ and 
\begin{align*} 
J_*(z_1,z_2) & \ge \int _{\{|z_2|<|z|<|z_1|\}} |\vphi(x+iy)| dx\tinyhsp dy 
= \int_{|z_2|}^{|z_1|} \int_{0}^{2\pi} |\vphi(r e^{i\theta})| r d\theta \tinyhsp dr \\
&\ge 2\pi \left| \frac{z_1}{z_1-z_2} \right| \int_{|z_2|}^{|z_1|} \frac{dr}{r} 
= 2\pi \frac{1}{|1- \frac{z_2}{z_1}|} \log \frac {\,|z_1|\,}{|z_2|}.    
\end{align*}
The equality for $\overline{K}_f$ is obvious from 
$\left|1+ e^{i\theta} \mu_f(z) \frac{\vphi(z)}{|\vphi(z)|}\right|^2
=(1-|\mu_f(z)|^2) + 2\Re \left(e^{i\theta} \mu_f(z) \frac{\vphi(z)}{|\vphi(z)|}\right) + 2|\mu_f(z)|^2$.  
The last inequality follows from 
$2\frac{|\mu_f(z)|+|\mu_f(z)|^2}{1-|\mu_f(z)|^2} 
= \frac{1+|\mu_f(z)|}{1-|\mu_f(z)|}-1 \le K_f(z)-1$.  
\end{proof}

\section{Main Theorem and Proof of Theorems \ref{thm:GutMartio} and \ref{thm:HoelderConformal}}

Our criterion for the pointwise conformality is as follows: 
\begin{thm} \label{thm:MainThm} 
Let $f: \C \to \C$ be a $K$-quasiconformal mapping with $f(0)=0$ and suppose that there exists 
$0<\delta <1$  such that 
\begin{equation}  \label{JtendsTo0} 
J(\mu_f; z_1,z_2) \to 0 \ \text{ when $z_1$ and $z_2$ tend to $0$ satisfying } 0<|z_2| \le \delta |z_1|.
\end{equation}
Then $f$ is conformal at $z=0$.  Moreover there exists a constant $C>0$ depending only on $K$ such that 
\begin{equation} \label{eq:ConformalityEstimate} 
\left| \log \frac{f(z)}{z} - \log f'(0) \right|_{\calC} 
\le C \liminf_{z_2 \to 0} J(\mu_f; z,z_2). 
\end{equation}
\end{thm}

This follows from the following lemma.  
\begin{lem}[Key Inequality] \label{lem:KeyInequality} 
Given $K >1$, there exist $\delta_1<1$ and $C$ such that  
if $f:\C \to \C$ is a $K$-quasiconformal mapping with $f(0)=0$, then for $0<|z_2| \le \delta_1|z_1|$, 
\begin{equation}
\begin{aligned} \label{eq:KeyIneq}
\left| \log \frac{f(z_1)}{z_1} - \log \frac{f(z_2)}{z_2} \right|_{\calC} 
&\le 
C J(\mu_f; z_1,z_2).
\end{aligned}
\end{equation}
\end{lem} 

\begin{proof} Let $\delta_1<1$ and $C_1$ be as in Lemma \ref{lem:LogVsPoincareDist} for $L=\log K$.  
Then take $C=\frac{C_1 (1+\delta_1)}{2\pi}$.  
For two distinct points $z_1, z_2 \in \C \sminus \{0\}$ and consider the cross-ratios 
$\zeta_1=\frac{z_2}{z_1} = \Cr(z_1,z_2, \infty, 0)$ and $\zeta_2=\frac{f(z_2)}{f(z_1)} = \Cr(f(z_1),f(z_2), \infty, 0)$.  
By Theorem \ref{thm:FundIneq} and Lemma \ref{lem:JandKbar} and $\log(1+x) \le x \ (x \ge0)$, we have 
\begin{equation} \label{eq:PoincareDistEstimate}
d_\Omega(\zeta_1, \zeta_2) \le \log \overline K_f(z_1,z_2,\infty, 0) 
\le 
 \frac {J(\mu_f; z_1,z_2)}{J_*(z_1,z_2)} 
\le \frac{J(\mu_f; z_1,z_2)}{2 \pi \frac{1}{\,|1- \frac{z_2}{z_1}|\,}  \log \frac {\,|z_1|\,}{|z_2|}}.
\end{equation}
By the classical Gr\"otzsch inequality, we have $d_\Omega(\zeta_1, \zeta_2) \le \log K=L$.  
Hence by Lemma \ref{lem:LogVsPoincareDist}, 
if $0<|z_2| \le \delta_1|z_1|$, then we have: 
\begin{equation}
\begin{aligned} \label{eq:KeyIneq}
\left| \log \frac{f(z_1)}{z_1} - \log \frac{f(z_2)}{z_2} \right|_{\calC} 
= \left| \log \frac {\zeta_1}{\zeta_2} \right|_{\calC} 
&\le C_1 d_\Omega(\zeta_1, \zeta_2) \log \frac{1}{|\zeta_1|} 
\le
\frac{C_1 \left|1-\frac{z_2}{z_1} \right|}{2 \pi} J(\mu_f; z_1,z_2) \\
&\le \frac{C_1(1+\delta_1)}{2 \pi} 
J(\mu_f; z_1,z_2) = C J(\mu_f; z_1,z_2). 
\end{aligned}
\end{equation} \vskip -5mm
\end{proof}

\begin{proof}[Proof of Theorem \ref{thm:MainThm}] 
This is an immediate consequence of Lemma \ref{lem:KeyInequality}.  
By the assumption \eqref{JtendsTo0}, in which we may replace $\delta$ by a smaller one so that $\delta \le \delta_1$, 
(c) of Lemma \ref{lem:CharacterizeConformality} holds, 
hence $f$ is conformal at $z=0$.  
Moreover taking the limit $z_2 \to 0$ in \eqref{eq:KeyIneq}, we obtain \eqref{eq:ConformalityEstimate}.  
\end{proof}

In order to deduce Theorem \ref{thm:GutMartio} from Theorem \ref{thm:MainThm}, we need to 
relate $J(\mu_f; z_1,z_2)$ to \eqref{eq:muSquareIntegrability} and \eqref{eq:muIntegrability}.  
For this purpose, we define the following quantity.  
\begin{defn}
Let $p>2$ and $p>s>0$.  For $\mu \in L^\infty(\C)$ with $||\mu||_\infty<1$, define 
\begin{equation} \label{eq:DefnIps} 
I_{p,s}(\mu; r) =  \iint_{\C} \frac{|\mu(z)|^p}{(1-|\mu(z)|^2)^p} \frac{dx \tinyhsp dy}{|z|^2 \left(1+ \frac{|z|}{r}\right)^s}
\end{equation}
For $0<r<R$, denote $A(r,R)=\{z \in \C: \, r<|z|<R\}$.  
\end{defn}

The following two lemmas will be proved in \S \ref{sec:EstimateOnJandIps}.

\begin{lem} \label{lem:EstimateOnJ} 
Let $\mu \in L^\infty(\C)$ with $||\mu||_\infty<1$.  Then for any $p>s>0$ with $p>2$ and $0<\rho<1$, 
there exists $C'=C'(p,s,\rho)>0$ such that if $0<|z_2|<\rho^2 |z_1|$, then 
\begin{align} \label{eq:EstimateOnJ1} 
\left| \iint_\C \frac{\mu(z)\vphi_{z_1,z_2 \vphantom{|}}(z)}{1-|\mu(z)|^2}dx \tinyhsp dy\right|
&\le \frac{1}{1-\rho^2}\left| \iint_{A(\rho^{-1} |z_2|, \rho |z_1|)} 
\frac{\mu(z)}{1-|\mu(z)|^2} \, \frac{dx\tinyhsp dy}{z^2} \right| + C' I_{p,s}(\mu; |z_1|)^{\frac{1}{p}}, \\ 
\label{eq:EstimateOnJ2} 
\iint_\C \frac{|\mu(z)|^2 |\vphi_{z_1,z_2 \vphantom{|}}(z)|}{1-|\mu(z)|^2}dx \tinyhsp dy 
&\le \frac{1}{1-\rho^2}\iint_{A(\rho^{-1} |z_2|, \rho |z_1|)} 
\frac{|\mu(z)|^2}{1-|\mu(z)|^2} \, \frac{dx\tinyhsp dy}{|z|^2} 
+ C' I_{p,s}(\mu; |z_1|)^{\frac{1}{p}}.
\end{align}
\end{lem}

\begin{lem} \label{lem:EstimateOnIps}
For $\mu \in L^\infty(\C)$ with $||\mu||_\infty<1$ satisfying \eqref{eq:muSquareIntegrability} and for $p>2$ and $p>s>0$, 
the integral $I_{p,s}(\mu; r)$ is finite.  
Moreover there exist constants $C_2$ and $C_3$ depending only on 
$K=\frac{1+||\mu||_\infty}{1-||\mu||_\infty}$ such that  for $0<r<r'$,  
\begin{equation} \label{eq:EstimateOnIps} 
I_{p,s}(\mu; r) \!= \!\! \iint_{\C} \!\frac{|\mu(z)|^p}{(1-|\mu(z)|^2)^p} \frac{dx \tinyhsp dy}{|z|^2 \left(1+ \frac{|z|}{r}\right)^{\!s}} 
\le C_2 \!\! \iint_{\{|z|<r'\} } \frac{|\mu(z)|^2}{1-|\mu(z)|^2} \frac{dx \tinyhsp dy}{|z|^2} 
+ \frac{C_3}{s} \!\! \left(\frac{r}{r'}\right)^{\!\!s} \!\!.
\end{equation}
Therefore $I_{p,s}(\mu; r) \to 0$ as $r \searrow 0$. 
\end{lem}

Assuming these lemmas, we can give: 

\begin{proof}[Proof of Theorem \ref{thm:GutMartio}]
Since the convergence in  \eqref{eq:muIntegrability} and  \eqref{eq:muSquareIntegrability} imply 
that the first terms on the right hand sides of \eqref{eq:EstimateOnJ1} and  \eqref{eq:EstimateOnJ2} 
tend to $0$ as $z_1 \to 0$, 
Theorem \ref{thm:GutMartio} follows from Theorem \ref{thm:MainThm} and Lemmas \ref{lem:EstimateOnJ}, \ref{lem:EstimateOnIps}.  
\end{proof}

\begin{proof}[Proof of Theorem \ref{thm:HoelderConformal}]
Suppose $I(r) =O(r^\beta)$ ($r \searrow 0$) and $0<\alpha<\frac{\beta}{2+\beta}$.  
According to Theorem \ref{thm:MainThm} and Lemma \ref{lem:EstimateOnJ}, in order to prove \eqref{eq:HoelderConformality}, 
it suffices to show that all the terms in \eqref{eq:EstimateOnJ1} and \eqref{eq:EstimateOnJ2} have order $O(r^\alpha)$.  
This is obvious for the first terms.  
Choose $s=2$ and $p>2$ so that $\beta<\frac{2\beta}{2+\beta} \frac{1}{p}$.  
Let $\gamma=\frac{2}{2+\beta}$ and take $r'=r^\gamma$ in Lemma \ref{lem:EstimateOnIps}.  
Both terms on the right hand side of \eqref{eq:EstimateOnIps} have order $O(r^{\frac{2\beta}{2+\beta}})$, hence 
$I_{p,s}(\mu; r)^{\frac{1}{p}} = O(r^\alpha)$.  Thus \eqref{eq:HoelderConformality} is proved.  
\end{proof}

\section{Estimates on the integrals $J(\mu; z_1,z_2)$ and $I_{p,s}(r)$} 
\label{sec:EstimateOnJandIps}
\begin{proof} [Proof of Lemma \ref{lem:EstimateOnJ}] 
Let $\mu, p, s, \rho$ be as in Lemma \ref{lem:EstimateOnJ} and suppose $0<|z_2| \le \rho^2 |z_1|$. 
Since 
\begin{gather} \label{eq:QuadDiffDecomposition} 
\vphi_{z_1,z_2 \vphantom{|}}(z) + \frac{z_1}{z_1+z_2} \cdot \frac{1}{z^2} 
= \frac{z_1}{z_1+z_2}(\psi_1(z)+\psi_2(z)), 
\end{gather}
where 
$\psi_1(z)= \frac{1}{(z-z_1)(z-z_2)}$ and $\psi_2(z)= \frac{z_1 z_2}{z^2(z-z_1)(z-z_2)}$, 
the decomposition of the integral into $\D(\rho^{-1} |z_2|)= \{z \in \C: |z|<\rho^{-1} |z_2|\}$, 
$\D^*(\rho |z_1|)=\{z \in \C: |z|>\rho |z_1|\}$ 
and $A(\rho^{-1} |z_2|,  \rho |z_1|)$ gives  
\begin{equation} \label{eq:DecomposeIntegral} 
\begin{aligned}
&\left| \iint_{\C} \frac{\mu(z)\vphi_{z_1,z_2 \vphantom{|}}(z)}{1-|\mu(z)|^2}dx \tinyhsp dy 
+ \frac{z_1}{z_1+z_2}  \iint_{A(\rho^{-1} |z_2|,  \rho |z_1|)} \frac{\mu(z)}{1-|\mu(z)|^2} \frac{dx \tinyhsp dy}{z^2} \right| \\
\le &
\left| \iint_{\D(\rho^{-1} |z_2|)} \frac{\mu(z)\vphi_{z_1,z_2 \vphantom{|}}(z)}{1-|\mu(z)|^2}dx \tinyhsp dy \right| 
+ \left| \iint_{\D^*(\rho |z_1|)} \frac{\mu(z)\vphi_{z_1,z_2 \vphantom{|}}(z)}{1-|\mu(z)|^2}dx \tinyhsp dy \right| \\
& \ + \frac{1}{|1+\frac{z_2}{z_1}|} \left(\left| \iint_{A(\rho^{-1} |z_2|,  \rho |z_1|)} \frac{\mu(z)\psi_1(z)}{1-|\mu(z)|^2}dx \tinyhsp dy \right| 
+ \left| \iint_{A(\rho^{-1} |z_2|,  \rho |z_1|)} \frac{\mu(z)\psi_2(z)}{1-|\mu(z)|^2}dx \tinyhsp dy \right| \right).  
\end{aligned}
\end{equation}
It is easy to see that \eqref{eq:EstimateOnJ1} holds 
if one can prove that each term on the right hand side of \eqref{eq:DecomposeIntegral} is bounded by $I_{p,s}(\mu; |z_1|)^{\frac{1}{p}}$ up to a constant factor.  
Take $q$ such that $\frac{1}{p}+\frac{1}{q}=1$, then $1<q<2$.  
For any measurable set $D \subset \C$, denote by $I_{p,s}(\mu; r, D)$ the integral in \eqref{eq:DefnIps} with the 
domain $\C$ replaced by $D$.  
For a measurable set $D \subset \C$ and an integrable function $\psi(z)$ on $D$, 
the H\"older inequality yields 
\begin{equation} \label{eq:HoelderIneq}
\begin{aligned}
\left| \iint_{D} \frac{\mu(z)\psi(z)}{1-|\mu(z)|^2}dx \tinyhsp dy \right| 
&\le 
\iint_{D} \Biggl| \frac{\mu(z)}{(1-|\mu(z)|^2) |z|^{\frac{2}{p}} \left(1+ \frac{|z|}{r}\right)^{\frac{s}{p}}} \Biggr| 
\cdot \biggl| |z|^{\frac{2}{p}} \left(1+ \frac{|z|}{r}\right)^{\frac{s}{p}} \psi(z) \biggr|  dx \tinyhsp dy \\
&\le 
I_{p,s}(\mu; r, D)^{\frac{1}{p}}
H(\psi, r, D)^{\frac{1}{q}}, 
\end{aligned}
\end{equation}
where $H(\psi, r, D)= \displaystyle \iint_{D}  |z|^{2q-2} \left(1+ \frac{|z|}{r}\right)^{s(q-1)} |\psi(z)|^q dx \tinyhsp dy$. 
In order to estimate the terms in \eqref{eq:DecomposeIntegral}, we apply \eqref{eq:HoelderIneq} with 
$r=|z_1|$, $\psi = \vphi_{z_1,z_2 \vphantom{|}}$, $\psi_1$, $\psi_2$,  and $D=\D(\rho^{-1} |z_2|)$, $\D^*(\rho |z_1|))$, $A(\rho^{-1} |z_2|, \rho |z_1|)$.  
It suffices to show that the corresponding $H(\psi, r, D)$ is finite.  
\par 
For the first term of the right hand side of \eqref{eq:DecomposeIntegral}, 
we now give an estimate on $H(\psi, r, D)$ for 
$\psi(z)=\vphi_{z_1,z_2 \vphantom{|}}(z)=-\frac{1}{z(z-z_2)(1-z/z_1)}$,  $r=|z_1|$ and $D=\D(\rho^{-1} |z_2|)$.  
By the change of variable $z=z_2 \zeta$, $\zeta=\xi + i \eta$, we have 
\begin{equation}
\begin{aligned} \label{eq:TopEndEstimate}
H(\vphi_{z_1,z_2 \vphantom{|}}, |z_1|, \D(\rho^{-1} |z_2|))
&= 
\iint_{ \{ |\zeta| < \frac{1}{\rho} \} } \frac{ |z_2|^{2q-2} |\zeta|^{2q-2} \left(1+ \frac{|z_2| |\zeta|}{|z_1|}\right)^{s(q-1)}}
{|z_2 \zeta (z_2 \zeta -z_2)(1- \frac{z_2 \zeta}{z_1})|^q}  |z_2|^2 d\xi \tinyhsp d\eta \\
&\le 
\frac{(1+\rho)^{s(q-1)}}{(1-\rho)^q} \iint_{ \{ |\zeta| < \frac{1}{\rho} \} } \frac{  |\zeta|^{q-2} }{|\zeta -1|^q} d\xi \tinyhsp d\eta 
=: H_1(\rho).
\end{aligned}
\end{equation}
The last integral converges, because its integrand has order $|\zeta|^{q-2}$ near $\zeta=0$ with $q-2>-2$ 
and order $|\zeta-1|^{-q}$ near $\zeta=1$ with $-q>-2$.  
\par
Similarly, setting either $z=z_1 \zeta$ or $z=z_2 \zeta$, we have 
\begin{align} 
&\begin{aligned} \label{eq:BottomEndEstimate}
H(\vphi_{z_1,z_2 \vphantom{|}}, |z_1|, \D^*(\rho |z_1|))
&= 
\iint_{\{|\zeta|>\rho\}} \frac{ |z_1|^q |z_1|^{2q-2} |\zeta|^{2q-2} \left(1+  |\zeta| \right)^{s(q-1)}}
{|z_1^3 \zeta^2 (1 - \frac{z_2}{z_1 \zeta})(\zeta - 1)|^q}  |z_1|^2 d\xi \tinyhsp d\eta \\
&\le 
\frac{1}{(1-\rho)^q} \iint_{\{|\zeta|>\rho\}} \frac{ (1+ |\zeta|)^{s(q-1)} }{ |\zeta|^2 |\zeta -1|^q} d\xi \tinyhsp d\eta
=:H_2(\rho).
\end{aligned} \displaybreak[0] \\
&\begin{aligned} \label{eq:Middle1Estimate}
H( \psi_1, |z_1|, A(\rho^{-1} |z_2|, \rho |z_1|) ) 
&= 
\iint_{ \{ \frac{|z_2|}{\rho |z_1|} \le |\zeta| \le \rho \} } \frac{ |z_1|^{2q-2} |\zeta|^{2q-2} \left(1+  |\zeta| \right)^{s(q-1)}}
{|z_1^2 \zeta (1 - \frac{z_2}{z_1 \zeta})(\zeta - 1)|^q}  |z_1|^2 d\xi \tinyhsp d\eta \\
&\le 
\frac{(1+\rho)^{s(q-1)}}{(1-\rho)^q} \iint_{\{ |\zeta| \le \rho  \} } \frac{ |\zeta|^{q-2}}{  |\zeta -1|^q} d\xi \tinyhsp d\eta 
=:H_3(\rho).
\end{aligned} \displaybreak[0] \\
&\begin{aligned} \label{eq:Middle2Estimate}
H( \psi_2, |z_1|, A(\rho^{-1} |z_2|, \rho |z_1|) ) 
&= 
\iint_{ \{ \frac{1}{\rho} \le |\zeta| \le \frac{\rho |z_1|}{|z_2|}\} } \frac{ |z_2|^{3q-2} |\zeta|^{2q-2} \left(1+ \frac{|z_2| |\zeta|}{|z_1|}\right)^{s(q-1)}}
{|z_2^3 \zeta^2 (\zeta - 1)(1- \frac{z_2 \zeta}{z_1})|^q}  |z_2|^2 d\xi \tinyhsp d\eta \\
&\le 
\frac{(1+\rho)^{s(q-1)}}{(1-\rho)^{2q}} \iint_{ \{ \rho^{-1} \le  |\zeta| \} } \frac{ 1 }{|\zeta|^2 |\zeta -1|^q} d\xi \tinyhsp d\eta 
=: H_4(\rho).
\end{aligned}
\end{align}
Again the integrals $H_j(\rho)$ converge, 
for example for $j=2$, its integrand has order $|\zeta|^{s(q-1)-2-q}$ near $\zeta=\infty$ 
with $s(q-1)-2-q=\frac{q}{p}(s-p)-2 < -2$ 
and order $|\zeta-1|^{-q}$ near $\zeta=1$ with $-q>-2$.  
The cases of $j=3,4$ are left to the reader.  
\par 
Thus by \eqref{eq:DecomposeIntegral}, \eqref{eq:HoelderIneq} and \eqref{eq:TopEndEstimate}--\eqref{eq:Middle2Estimate}, 
there exists $C'=C'(p, s, \rho)$ such that  \eqref{eq:EstimateOnJ1} holds.  
For \eqref{eq:EstimateOnJ2}, replace $\mu(z)$ in the numerator by 
$|\mu(z)|^2 |\vphi_{z_1,z_2 \vphantom{|}}(z)|/\vphi_{z_1,z_2 \vphantom{|}}(z)$ 
and use $|\mu(z)|^2 \le |\mu(z)|$ to obtain similar estimates.  
In fact, we can use the same constant $C'$.  
Thus Lemma \ref{lem:EstimateOnJ} is proved.  
\end{proof}

\begin{rem} If we assume \eqref{eq:TWBL} in Theorem \ref{thm:TWB}, 
then it is also possible to show \eqref{JtendsTo0} by estimating $\iint (K_f(z) -1) |\vphi_{z_1,z_2 \vphantom{|}}(z)| dx \tinyhsp dy$ 
which is divided into several regions defined by $|z| \ge \rho^{-1} |z_1|$, $|z-z_1| \le \rho |z_1|$, $|z-z_2| \le \rho |z_2|$ and the rest, 
where $0<\rho <1$ will need to be chosen small according to the target $\eps$. 
\end{rem}

\begin{proof} [Proof of Lemma \ref{lem:EstimateOnIps}] 
Since $K+\frac{1}{K}+2=\frac{4}{1-||\mu||_\infty^2}$ and $K-\frac{1}{K}=\frac{4||\mu||_\infty}{1-||\mu||_\infty^2}$, 
the integrand in $I_{p,s}(\mu; r)$ is bounded by both 
\begin{equation*}
\left( \frac{K+\frac{1}{K}+2}{4} \right) \! \left( \frac{K-\frac{1}{K}}{4} \right)^{p-2} \! \frac{|\mu(z)|^2}{1-|\mu(z)|^2}\, \frac{1}{|z|^2}
\quad \text{  and  } \quad 
\left( \frac{K-\frac{1}{K}}{4} \right)^{p}\!  \frac{1}{|z|^2 \left( \frac{|z|}{r} \right)^s}. 
\end{equation*}
Integrating over $\{|z|<r'\}$ and $\{|z| \ge r'\}$, we immediately obtain \eqref{eq:EstimateOnIps}.  
Hence $I_{p,s}(\mu; r)$ is finite by the assumption \eqref{eq:muSquareIntegrability}.  
\par 
One can make the first term of the right hand side of \eqref{eq:EstimateOnIps} small by choosing $r'$ small, then make the second term small 
by choosing $r$ even smaller.  
Therefore $\lim_{r \searrow 0} I_{p,s}(\mu; r) =0$.  
\end{proof}

\appendix
\section*{Appendex A. Proof of Theorem \ref{thm:FundIneq}: Gr\"otzsch-type inequality}

We prove Theorem \ref{thm:FundIneq} closely following Ahlfors \cite{AhlforsQC} Chap. III.D, 
but improving the detail.  The difference is that we do not replace $|1+\tilde{\mu}|$ by $1+|\tilde{\mu}|$ 
in \eqref{eq:LengthAreaIneqOnTorus} below.  
\par 
Given $(z_1, z_2, z_3, z_4)$, there exist $\tau \in \C$ with $\Im \tau>0$ and 
a holomorphic branched double covering $p: \E_{\tau} \to \CC$ 
branching over these four points, where $\E_{\tau} = \C/(\Z+\Z \tau)$.  
(If $z_4=\infty$, $p$ can be taken as the Weierstrass $\wp$-function.) 
Then $p(w)$ satisfies $p'(w)^2 = c \prod_{j} (p(w) -z_j) = \frac{c}{\vphi(p(w))}$ for some $c \in \C \sminus \{0\}$. 
There exist a counterpart $\hat{p}: \E_{\tau'} \to \CC$ for $(z_1', z_2', z_3', z_4')$, and a lift $\tilde{f}: \E_{\tau} \to \E_{\tau'}$, 
sending generators $1, \tau$ to $1, \tau'$ and satisfying $f \circ p = \hat{p} \circ \tilde{f}$.  
For $z=\vphi(w)$, $\mu=\mu_f$, $\tilde{\mu} = \mu_{\tilde{f}}$, we have 
$|\vphi(z)| dx\tinyhsp dy = |\vphi(p(w))| |p'(w)|^2 du\tinyhsp dv= |c| \tinyhsp du \tinyhsp dv$ and 
$\tilde{\mu}(w) \frac{c}{|c|} = \mu(p(w)) \frac{\, \overline{p'(w)} \,}{p'(w)} \cdot \frac{ p'(w)^2 \vphi(p(w)) }{| p'(w)^2 \vphi(p(w)) |} 
= \mu(z) \frac{\vphi(z)}{|\vphi(z)|}$. 
Therefore the double cover $p$ gives 
\begin{equation} \label{eq:FundIneqTorus}
\overline{K}_f(z_1,z_2,z_3,z_3) =  \sup_{\theta \in \R} K_{\tilde{f}, \theta}, 
\quad \text{ where } 
K_{\tilde{f}, \theta} = 
\frac{ \iint_{\E_{\tau}} \frac {\, \left|1+ e^{i\theta} \tilde{\mu}(w) \right|^2}{1-|\tilde{\mu}(w)|^2} du \tinyhsp dv}
{\Area(\E_{\tau})}, 
\end{equation}
and $\Area(\E_{\tau})=\iint_{\E_{\tau}} du \tinyhsp dv=\Im \tau$.  
\par 
We now follow the standard Gr\"otzsch argument: 
The map $\tilde{f}$ sends each horizontal curve on $\E_{\tau}$ to a closed curve homotopic to 
a horizontal curve in $\E_{\tau'}$.  Since $\tilde{f}$ is absolutely continuous along almost all horizontal lines, we have $\int_0^1 |\tilde{f}_u(u+iv)|du \ge 1$ for a.a. $v$.  
By integrating over $v \in [0, \Im \tau]$ and using $\tilde{f}_u=\tilde{f}_w+\tilde{f}_{\overline{w}}=(1+\tilde{\mu}) \tilde{f}_w$, we have 
\begin{equation*}
\Im \tau \le \iint_{\E_{\tau}} |(1+\tilde{\mu}) \tilde{f}_w | du \tinyhsp dv.
\end{equation*}
Cauchy-Schwarz inequality together with 
$\Jac \tilde{f} = |\tilde{f}_w|^2-|\tilde{f}_{\overline{w}}|^2 = |\tilde{f}_w|^2 (1-|\tilde{\mu}|^2)$ 
implies 
\begin{equation} \label{eq:LengthAreaIneqOnTorus}
\Im \tau^2 \le \iint_{\E_{\tau}} \Jac \tilde{f} \tinyhsp du \tinyhsp dv 
\iint_{\E_{\tau}} \frac{|1+\tilde{\mu}|^2}{1-|\tilde{\mu}|^2} du \tinyhsp dv
\le \Im \tau'  \iint_{\E_{\tau}} \frac{|1+\tilde{\mu}|^2}{1-|\tilde{\mu}|^2} du \tinyhsp dv.
\end{equation}
Hence we have $\Im \tau \le K_{\tilde{f}, 0} \, \Im\tau'$, which means that $\tau$ is not contained in the open horodisk 
which is tangent to $\partial \HH$ at $\infty$ and has distance $\log K_{\tilde{f}, 0}$ to $\tau'$. 
\par 
If we change the generators of $\Z +\Z\tau$ from $1, \tau$ to $c\tau+d$, $a\tau+b$ with 
$A= \left( \begin{smallmatrix} a & b \\ c & d \end{smallmatrix} \right) \in SL(2, \Z)$, then we obtain an estimate 
on $\Im \frac{a\tau+b}{c\tau+d}$ and $\Im \frac{a\tau'+b}{c\tau'+d}$, 
and it has an effect of rotating the horizontal axis and 
$\tilde{\mu}$ in the integral should be replaced by $e^{i \theta}\tilde{\mu}$ for some $\theta \in \R$.  
Thus we obtain $\Im \frac{a\tau+b}{c\tau+d} \le K_{\tilde{f}, \theta} \, \Im \frac{a\tau'+b}{c\tau'+d}$, 
which means  $\tau$ is not in the open horodisk which is tangent to $\partial \HH$ at $A^{-1}(\infty)= - \frac{d}{c}$ 
and has distance $\log K_{\tilde{f}, \theta}$ to $\tau'$.  
If $\tau$ had distance greater than $\log \sup_{\theta \in \R} K_{\tilde{f}, \theta}$ from $\tau'$, 
then $\tau$ would be in one of the horodisks as above, 
because $A^{-1}(\infty)$ ($A \in SL(2, \Z)$) are dense on $\partial \HH$.  
Hence we conclude that $d_{\HH}(\tau, \tau') \le \log \overline{K}_f(z_1,z_2,z_3,z_3)$.  
Finally the cross-ratio $\Cr(z_1, z_2, z_3, z_4)=\lambda(\tau)$ as a function of $\tau$ is 
the elliptic modular function $\lambda: \HH \to \Omega$, 
which is a universal covering map.  
(See \cite{AhlforsQC}.) 
Therefore $d_{\Omega}(\lambda(\tau), \lambda(\tau')) \le  d_{\HH}(\tau, \tau')$ and 
Theorem \ref{thm:FundIneq} is proved.  
\hfill \qed



\begin{thebibliography}{MM2}
\bibitem[A1]{AhlforsQC} Ahlfors, L. V., 
	Lectures on quasiconformal mappings. Second edition. 
	With supplemental chapters by C. J. Earle, I. Kra, M. Shishikura and J. H. Hubbard. 
	University Lecture Series, 38. American Mathematical Society, Providence, RI, 2006.

\bibitem[A2]{AhlforsCI} -----, 
	Conformal invariants: topics in geometric function theory. 
	McGraw-Hill, 1973.


\bibitem[B]{Belinskii}
Belinski\u \i, P. P., Behavior of a quasiconformal mapping at an isolated singular point (Russian), L'vov Gos. Univ. Uchen. Zap., Ser. Meh.-Mat. no. 6, 29 (1954), 58--70. \\
Belinski\u \i, P. P., General Properties of Quasiconformal Mappings (Russian), Izdat. ``Nauka'' Sibirsk. Otdel., Novosibirsk, 1974.


\bibitem[BJ]{BrakalovaJenkins} Brakalova, M. and Jenkins, J. A., 
	On the local behavior of certain homeomorphisms. Kodai Math. J. 17 (1994), no. 2, 201--213.
 
\bibitem[D]{Drasin} Drasin, D., 
	On the Teichm\"uller-Wittich-Belinski\u \i theorem. 
	Results Math. 10 (1986), no. 1-2, 54--65. 


\bibitem[GL]{GardinerLakic} Gardiner, F. P. and Lakic N., 
	Quasiconformal Teichm\"uller theory.  
	Mathematical Surveys and Monographs, 76. 
	American Mathematical Society, 2000.

\bibitem[GM]{GutMartio} Gutlyanski\u\i, V. and Martio, O., 
	Conformality of a quasiconformal mapping at a point,  
	J. Anal. Math. 91 (2003), 179--192.

\bibitem[HSS]{HSS} 
	Hubbard, J. H., Schleicher, D. and Shishikura, M., 
	Exponential Thurston maps and limits of quadratic differentials. J. Amer. Math. Soc. 22 (2009), no. 1, 77--117.
	
\bibitem[L1]{LehtoDiff} Lehto, O., 
	On the differentiability of quasiconformal mappings with prescribed complex dilatation. 
	Ann. Acad. Sci. Fenn. Ser. A I No. 275 (1960) 1-28 pp. 

\bibitem[LV]{LehtoVirtanen} Lehto O. and Virtanen, K. I., 
	Quasiconformal Mappings in the Plane, 2nd Edn., Springer-Verlag, Berlin, 1973.

\bibitem[McM]{McMullenRenorm2} McMullen, C. T., 
	Renormalization and 3-manifolds which fiber over the circle,  
	Annals of Mathematics Studies, 142. 
	Princeton University Press, 1996.


\bibitem[RW]{ReichWalczak} Reich, E. and Walczak, H. R., 
	On the behavior of quasiconformal mappings at a point. Trans. Amer. Math. Soc. 117 1965 338--351.


\bibitem[S]{Schatz} Schatz, A., 
	On the local behavior of homeomorphic solutions of Beltrami's equations. 
	Duke Math. J. 35 1968 289--306. 
	
\bibitem[T]{Teichmuller} Teichm\"uller, O., 
	Untersuchungen \"uber konforme und quasikonforme Abbildung, 
	Deutsche Math. 3 (1938), 621--678.

\bibitem[W]{Wittich} Wittich, H., 
	Zum Beweis eines Satzes \"uber quasikonforme Abbildungen, 
	Math. Z. 51 (1948), 275--288.

\end{thebibliography}
\end{document}